\documentclass{article}
\usepackage{geometry,graphicx,latexsym,amssymb,verbatim,amsthm}
\usepackage{tikz}
\usetikzlibrary{calc}
\usepackage{hyperref}

\newtheorem{thm}{Theorem}
\newtheorem{ob}[thm]{Observation}
\newtheorem{lem}[thm]{Lemma}
\newtheorem{prop}[thm]{Proposition}

\newtheorem{defin}[thm]{Definition}
\newtheorem{conj}[thm]{Conjecture}

\newtheorem{clm}{Claim}[thm]

\newcommand{\barD}{\overline{D}}
\newcommand{\barS}{\overline{S}}
\newcommand{\cP}{\mathcal{P}}
\newcommand{\cA}{\mathcal{A}}
\newcommand{\cT}{\mathcal{T}}
\newcommand{\cC}{\mathcal{C}}

\newcommand{\1}{\vspace{0.1cm}}

\newcommand{\gL}{\gamma_L}

\newcommand{\smallqed}{{\tiny ($\Box$)}}

\newcommand{\claimproof}{\noindent\emph{Proof of claim.} }

\begin{document}

\title{Locating-dominating sets in twin-free graphs}

\author{Florent Foucaud\footnote{\noindent Department of Mathematics, University of Johannesburg, Auckland Park, 2006, South Africa.}~\footnote{\noindent LIMOS, UMR CNRS 6158, Universit\'e Blaise Pascal, Clermont-Ferrand, France.} \and Michael A. Henning\footnotemark[1] \and Christian L\"{o}wenstein\footnote{\noindent Institute of Optimization and Operations Research, Ulm University, Ulm 89081, Germany.} \and Thomas Sasse\footnotemark[3]}

\date{May 19, 2015}

\maketitle

\begin{abstract}
A locating-dominating set of a graph $G$ is a dominating set $D$ of $G$ with the additional property that every two distinct vertices outside $D$ have distinct neighbors in $D$; that is, for distinct vertices $u$ and $v$ outside $D$, $N(u) \cap D \ne N(v) \cap D$ where $N(u)$ denotes the open neighborhood of $u$. A graph is twin-free if every two distinct vertices have distinct open and closed neighborhoods. The location-domination number of $G$, denoted $\gL(G)$, is the minimum cardinality of a locating-dominating set in $G$. It is conjectured [D. Garijo, A. Gonz\'alez and A. M\'arquez. The difference between the metric dimension and the determining number of a graph. \textit{Applied Mathematics and Computation} \textbf{249} (2014), 487--501] that if $G$ is a twin-free graph of order $n$ without isolated vertices, then $\gL(G)\le \frac{n}{2}$. We prove the general bound $\gL(G)\le \frac{2n}{3}$, slightly improving over the $\lfloor\frac{2n}{3}\rfloor+1$ bound of Garijo et al. We then provide constructions of graphs reaching the $\frac{n}{2}$ bound, showing that if the conjecture is true, the family of extremal graphs is a very rich one. Moreover, we characterize the trees $G$ that are extremal for this bound. We finally prove the conjecture for split graphs and co-bipartite graphs.
\end{abstract}


\section{Introduction}

A \emph{dominating set} in a graph $G$ is a set $D$ of vertices of $G$ such that every vertex outside $D$ is adjacent to a vertex in $D$. The \emph{domination number}, $\gamma(G)$, of $G$ is the minimum cardinality of a dominating set in $G$. The literature on the subject of domination parameters in graphs up to the year 1997 has been surveyed and detailed in the two books~\cite{hhs1, hhs2}.
Among the existing variations of domination, the one of \emph{location-domination} is widely studied. A \emph{locating}-\emph{dominating set} is a dominating set $D$ that locates/distinguishes all the vertices in the sense that every vertex not in $D$ is uniquely determined by its neighborhood in $D$. The \emph{location-domination number} of $G$, denoted $\gL(G)$, is the minimum cardinality of a locating-dominating set  in $G$. The concept of a locating-dominating set was introduced and first studied by Slater~\cite{s2,s3} and studied in~\cite{cst,fh,rs,s2,s3,s4} and elsewhere.

A classic result due to Ore~\cite{o9} states that every graph without isolated vertices has a dominating set of cardinality at most half its order. While there are many graphs (without isolated vertices) which have location-domination number much larger than half their order, the only such graphs that are known contain many \emph{twins}, that is, pairs of vertices with the same closed or open neighborhood. It was therefore recently conjectured by Garijo et al.~\cite{conjpaper} that in the absence of twins, the classic bound of one-half the order for the domination number also holds for the location-domination number. In this paper, we continue the study of~\cite{conjpaper} by proving this conjecture for two standard graph classes: split graphs and co-bipartite graphs. We also describe interesting families of graphs that would, if the conjecture is true, provide extremal examples.

\noindent\textbf{Definitions and notations.} For notation and graph theory terminology, we in general follow~\cite{hhs1}. Specifically, let $G$ be a graph with vertex set $V(G)$, edge set $E(G)$ and with no isolated vertex. The \emph{open neighborhood} of a vertex $v \in V(G)$ is $N_G(v) = \{u \in V \, | \, uv \in E(G)\}$ and its \emph{closed neighborhood} is the set $N_G[v] = N_G(v) \cup \{v\}$. For a set $S$ of vertices of $G$, $N_G[S]$ is the union of all closed neighborhoods of vertices in $S$. The \emph{degree} of $v$ is $d_G(v) = |N_G(v)|$. If the graph $G$ is clear from the context, we simply write $V$, $E$, $N(v)$, $N[v]$, $N[S]$ and $d(v)$ rather than $V(G)$, $E(G)$, $N_G(v)$, $N_G[v]$, $N_G[S]$ and $d_G(v)$, respectively. 
%

A set $D$ is a dominating set of $G$ if $N[v] \cap D \ne \emptyset$ for every vertex $v$ in $G$, or, equivalently, $N[D] = V(G)$. Two distinct vertices $u$ and $v$ in $V(G) \setminus D$ are \emph{located} (or \emph{distinguished}) by $D$ if they have distinct neighbors in $D$; that is, $N(u) \cap D \ne N(v) \cap D$. If a vertex $u \in V(G) \setminus D$ is located from every other vertex in $V(G)\setminus D$, we simply say that $u$ is \emph{located} (or \emph{distinguished}) by $D$.

A set $S$ is a \emph{locating set} of $G$ if every two distinct vertices outside $S$ are located by $S$. In particular, if $S$ is both a dominating set and a locating set, then $S$ is a locating-dominating set.
We remark that the only difference between a locating set and a locating-dominating set in $G$ is that a locating set might have a unique non-dominated vertex.

An \emph{independent set} in $G$ is a set of vertices no two of which are adjacent, and a \emph{clique} is a set of vertices every two of which are adjacent. The \emph{independence number} and the \emph{clique number} of $G$ are the maximum cardinality of an independent set and a clique in $G$, respectively.  The complement of an independent set in $G$ is a \emph{vertex cover} in $G$. Thus if $S$ is a vertex cover in $G$, then every edge of $G$ is incident with at least one vertex in $S$.
Two edges in a graph $G$ are \emph{independent} if they are vertex-disjoint in $G$. A set of pairwise independent edges of $G$ is called a \emph{matching} in $G$. A \emph{perfect matching} $M$ in $G$ is a matching such that every vertex of $G$ is incident to an edge of $M$.

Two distinct vertices $u$ and $v$ of a graph $G$ are \emph{open twins} if $N(u)=N(v)$ and \emph{closed twins} if $N[u]=N[v]$. Further, $u$ and $v$ are \emph{twins} in $G$ if they are open twins or closed twins in $G$. A graph is \emph{twin-free} if it has no twins.

A \emph{clique} in $G$ is a set of vertices that induce a complete subgraph.
A \emph{split graph} is a graph whose vertex set can be partitioned into an independent set and a clique, and a \emph{co}-\emph{bipartite graph} is a graph whose vertex set can be partitioned into two cliques.
%
%
We use the standard notation $[k] = \{1,2,\ldots,k\}$.

\noindent\textbf{Conjectures and known results.} The conjecture that motivates our study is stated as follows:

\begin{conj}[Garijo et al. \cite{conjpaper}]\label{conj}
Every twin-free graph $G$ of order $n$ without isolated vertices satisfies $\gL(G)\le \frac{n}{2}$.
\end{conj}

Using a proof technique based on matchings, the authors of~\cite{conjpaper} proved Conjecture~\ref{conj} for graphs without $4$-cycles (which include trees). They also proved that a vertex cover of a twin-free graph is a locating-dominating set.

\begin{prop}[Garijo et al. \cite{conjpaper}]\label{prop:VC}
If $G$ is a twin-free graph without isolated vertices, then every vertex cover of $G$ is also a locating-dominating set.
\end{prop}

As an immediate consequence of Proposition~\ref{prop:VC}, graphs with independence number at least half the order verify the conjecture. In particular, this is true for bipartite graphs.
Using the relation $\gL(G)\le \gL(\overline{G})+1$ relating the location-domination number of a graph $G$ and its complement $\overline{G}$ (discovered by Hernando et al.~\cite{hmp14}), Garijo et al.~\cite{conjpaper} observed that a twin-free graph of order~$n$ with clique number at least $\lceil\frac{n}{2}\rceil+1$ also satisfies the conjectured bound.

The best general upper bound to date is due to Garijo et al. \cite{conjpaper} who showed that $\gL(G)\le \lfloor\frac{2n}{3}\rfloor+1$ holds for every twin-free graph $G$ of order~$n$ and without isolated vertices.

In an earlier paper, Henning and L\"owenstein~\cite{hl12} proved that every connected cubic claw-free graph (not necessarily twin-free) has a locating-total dominating set\footnote{A locating-total dominating set $D$ is a locating-dominating set that is also a total domnating set, that is, every vertex of the graph has a neighbor in $D$.} of size half its order, which implies that Conjecture~\ref{conj} is true for such graphs. Moreover they conjectured this to be true for every connected cubic graph, with two exceptions --- which, if true, would imply Conjecture~\ref{conj} for cubic graphs.

\noindent\textbf{Our results.} We slightly improve the general bound of~\cite{conjpaper} by proving the bound $\gL(G)\le \frac{2n}{3}$ in Section~\ref{sec:general}. In Section~\ref{sec:examples}, we provide several constructions for graphs with location-domination number half their order and we characterize all trees for which the bound is tight. The variety of these constructions shows that these graphs have a rich structure, which is an indication that Conjecture~\ref{conj} might be difficult to prove. We then continue to give support to Conjecture~\ref{conj} by proving it for split graphs and co-bipartite graphs in Section~\ref{sec:cobip-split}.

\section{General bound}\label{sec:general}

The authors of~\cite{conjpaper} proved that every twin-free graph $G$ of order $n$ without isolated vertices satisfies $\gL(G)\le \lfloor\frac{2n}{3}\rfloor+1$. We slightly improve this bound in the following theorem. For this purpose, we shall need the following well-known property of minimum dominating sets in graphs first observed by Bollob\'{a}s and Cockayne~\cite{BC79indep}. Given a set $S$ in a graph $G$ and a vertex $v\in S$, an \emph{$S$-external private neighbor} of $v$ is a vertex outside $S$ that is adjacent to $v$ but to no other vertex of $S$ in $G$.

\begin{prop}[Bollob\'{a}s, Cockayne~\cite{BC79indep}]\label{p:BC}
If $G$ is a graph with no isolated vertex, then there exists a
minimum dominating set $S$ in $G$ with the property that every vertex of $S$ has an $S$-external private neighbor.
\end{prop}

\begin{thm}
If $G$ is a twin-free graph of order~$n$ with no isolated vertices, then $\gL(G) \le 2n/3$.
 \label{t:generalbd}
\end{thm}
\begin{proof}
For an arbitrary subset $S$ of vertices in $G$, let $\cP_S$ be a partition of $\barS = V(G) \setminus S$ with the property that all vertices in the same part of the partition have the same open neighborhood in $S$ and vertices from different parts of the partition have different open neighborhood in $S$. Let $|\cP_S| = k(S)$. Let $X_S$ be the set of vertices in $\barS$ that belong to a partition set in $\cP_S$ of size~$1$ and let $Y_S = \barS \setminus X_S$. Hence every vertex in $Y_S$ belongs to a partition set of size at least~$2$. Let $n_1(S) = |X_S|$ and let $n_2(S) = k(S) - n_1(S)$.
Let $S$ be a minimum dominating set in $G$ with the property that every vertex of $S$ has an $S$-external private neighbor. Such a set exists by Proposition~\ref{p:BC}. We note that $n_1(S) + n_2(S) \ge |S|$ since every vertex of $S$ has an external private neighbor.
Among all supersets $S'$ of $S$ with the property that $n_1(S') + n_2(S') \ge |S'|$, let $D$ be chosen to be inclusion-wise maximal. (Possibly, $D = S$.)

\begin{clm}\label{claimA}
The vertices in each partition set of size at least~$2$ in $\cP_D$ have distinct neighborhoods in $X_D$, and therefore $D \cup X_D$ is a locating-dominating set of $G$.
\end{clm}
\claimproof
Let $u$ and $v$ be two vertices that belong to a partition set $T$, of size at least~$2$ in $\cP_D$. Since $G$ is twin-free, there exists a vertex $w \notin \{u,v\}$ that is adjacent to exactly one of $u$ and~$v$.
Since $u$ and $v$ have the same neighbors in $D$, we note that $w \notin D$.
Hence, $w \in \barD = V(G) \setminus D$.
Suppose for a contradiction that $w \in Y_D$ and consider the set $D' = D \cup \{w\}$.
Let $R$ be an arbitrary partition set in $\cP_{D}$ that might or might not contain $w$.
If $w$ is either adjacent to every vertex of $R\setminus\{w\}$ or adjacent to no vertex in $R\setminus\{w\}$, then $R\setminus\{w\}$ is a partition set in $\cP_{D'}$.
If $w$ is adjacent to some, but not all, vertices of $R\setminus\{w\}$, then there is a partition $R\setminus\{w\} = (R_1,R_2)$ of $R\setminus\{w\}$ where $R_1$ are the vertices in $R\setminus\{w\}$ adjacent to $w$ and $R_2$ are the remaining vertices in $R\setminus\{w\}$.
In this case, both sets $R_1$ and $R_2$ form a partition set in $\cP_{D'}$.
In particular, we note that there is a partition $T\setminus\{w\} = (T_1,T_2)$ of $T\setminus\{w\}$ where both sets $T_1$ and $T_2$ form a partition set in $\cP_{D'}$.
Therefore, $n_1(D') + n_2(D') \ge n_1(D) + n_2(D) + 1 \ge |D| + 1 = |D'|$, contradicting the maximality of $D$. Hence, $w \notin Y_D$. Therefore, $w \in X_D$.
Hence, $u$ and $v$ are located by the set $X_D$ in $G$.~\smallqed

\medskip
Let $Y'_D$ be obtained from $Y_D$ by deleting one vertex from each partition set of size at least~$2$ in $\cP_D$, and let $D' = D \cup Y'_D$. Then, $|D'| = n - n_1(D) - n_2(D)$. By definition of the partition $\cP_D$, every vertex in $V(G) \setminus D'$ has a distinct nonempty neighborhood in $D$ and therefore in $D'$. Hence we have the following claim.

\begin{clm}\label{claimB}
The set $D'$ is a locating-dominating set of $G$.
\end{clm}

By Claim~\ref{claimA}, the set $D \cup X_D$ is a locating-dominating set of $G$ of cardinality~$|D| + n_1(D)$. By Claim~\ref{claimB}, the set $D'$ is a locating-dominating set of $G$ of cardinality~$n - n_1(D) - n_2(D)$. Hence,
\begin{equation}
\gL(G) \le \min \{ |D| + n_1(D), n - n_1(D) - n_2(D) \}.
\label{Eq1}
\end{equation}
Inequality~(\ref{Eq1}) implies that if $n - n_1(D) - n_2(D) \le \frac{2}{3}n$, then $\gL(G) \le 2n/3$. Hence we may assume that $n - n_1(D) - n_2(D) > \frac{2}{3}n$, for otherwise the desired upper bound on $\gL(G)$ follows. With this assumption, $n_1(D) + n_2(D) < \frac{1}{3}n$. By our choice of the set $D$, we recall that $|D| \le n_1(D) + n_2(D)$. Therefore,
\[
|D| + n_1(D) \le 2n_1(D) + n_2(D) \le 2(n_1(D) + n_2(D)) < \frac{2}{3}n.
\]
Hence, by Inequality~(\ref{Eq1}), $\gL(G) < 2n/3$. This completes the proof of Theorem~\ref{t:generalbd}.
\end{proof}

\section{Twin-free graphs with location-domination number half their order}\label{sec:examples}

We observe that every connected graph $G$ on four or six vertices has location-domination number at least half its order. This is clear if $G$ has four vertices. If $G$ has six vertices, then $\gL(G)\geq 3$. Indeed, suppose to the contrary that there is a locating-dominating set $D$ of size~$2$. Then, two vertices of $V(G)\setminus D$ can be dominated by a single vertex, and one, by two vertices. But then $G$ has at most five vertices, a contradiction. Hence, the class of twin-free graphs of order~$6$ already yields a simple example of graphs that are extremal with respect to Conjecture~\ref{conj}.

In the remaining part of this section, we provide infinite families of twin-free graphs with location-domination number half their order.

\subsection{Families of graphs with small domination number but location-domination number half the order}

Every twin-free graph with domination number half the order also has location-domination number at least half its order. These graphs are known to be exactly the graphs where each component is either a $4$-cycle (but then the graph is not twin-free), or it is a corona graph, that is, it has been obtained from any graph by adding a pending edge to each of its vertices~\cite{px82}. Indeed it is clear that in such graphs, every dominating set is also a locating-dominating set.
However, not all graphs with large location-domination number have large domination number. Perhaps the simplest class of connected twin-free graphs with large location-domination number but small domination number is the class of graphs constructed as follows. For $k \ge 3$, let $H_k$ be the graph obtained from $K_{2,k}$ by selecting one of the two vertices of degree~$k$ and subdividing every edge incident with it, and then adding a pendant edge to both vertices of degree~$k$. The resulting graph, $H_k$, has order $2k+4$, domination number~$2$, and location-domination number exactly one-half the order (namely, $k+2$). The graph $H_4$, for example, is illustrated in Figure~\ref{f:H4}, where the darkened vertices form a minimum locating-dominating set in $H_4$.

\begin{figure}[htb]
\tikzstyle{every node}=[circle, draw, fill=black!0, inner sep=0pt,minimum width=.16cm]
\begin{center}
\begin{tikzpicture}[thick,scale=.6]
  \draw(0,0) { 
    +(2.50,2.50) -- +(3.75,2.50)
    +(3.75,1.88) -- +(2.50,1.88)
    +(2.50,2.50) -- +(1.25,1.25)
    +(1.25,1.25) -- +(0.00,1.25)
    +(1.25,1.25) -- +(2.50,1.88)
    +(1.25,1.25) -- +(2.50,0.63)
    +(1.25,1.25) -- +(2.50,0.00)
    +(2.50,0.00) -- +(3.75,0.00)
    +(3.75,0.00) -- +(5.00,1.25)
    +(5.00,1.25) -- +(3.75,1.88)
    +(3.75,2.50) -- +(5.00,1.25)
    +(5.00,1.25) -- +(6.25,1.25)
    +(5.00,1.25) -- +(3.75,0.63)
    +(3.75,0.63) -- +(2.50,0.63)
    +(3.75,0.63) node{}
    +(2.50,0.63) node[fill=black!100]{}
    +(2.50,1.88) node[fill=black!100]{}
    +(2.50,2.50) node[fill=black!100]{}
    +(3.75,2.50) node{}
    +(3.75,1.88) node{}
    +(2.50,0.00) node[fill=black!100]{}
    +(3.75,0.00) node{}
    +(5.00,1.25) node[fill=black!100]{}
    +(1.25,1.25) node[fill=black!100]{}
    +(0.00,1.25) node{}
    +(6.25,1.25) node{}
  };
\end{tikzpicture}
\end{center}
\vskip -0.6 cm \caption{The graph $H_4$.} \label{f:H4}
\end{figure}
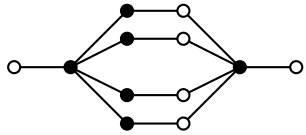

The following construction provides a family of dense twin-free graphs with domination number~2, but location-domination number one-half the order.

\begin{defin}[\cite{fgknpv11}]\label{def:Ak}
For an integer $k\ge  2$, let $A_k=(V_k, E_k)$ be the graph with
vertex set $V_k=\{x_1,\ldots,x_{2k}\}$ and edge set $E_k=\{x_{i}x_{j},
|i-j| \le  k-1\}$.
\end{defin}

All graphs of this family, defined in~\cite{fgknpv11} in the context of identifying codes, are twin-free and co-bipartite (and hence have domination number~2). In fact each graph $A_k$ is isomorphic to the $k$-th distance power of the path $P_{2k}$. See Figure~\ref{fig:A_k} for an illustration, where the darkened vertices form a minimum locating-dominating set in $A_k$.

    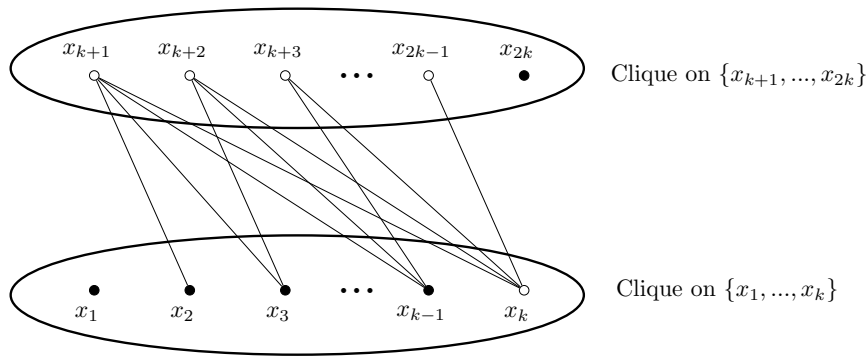
\begin{figure}[ht]
        \begin{center}
\scalebox{0.9}{\begin{tikzpicture}[join=bevel,inner sep=0.5mm,]
  \node (x_kplus1) at (90bp,200bp) [draw, circle] {};
  \node  at (87bp,210bp) {$x_{k+1}$};
  \node (x_kplus2) at (130bp,200bp) [draw, circle] {};
  \node  at (127bp,210bp) {$x_{k+2}$};
  \node (x_kplus3) at (170bp,200bp) [draw, circle] {};
  \node  at (167bp,210bp) {$x_{k+3}$};
  \node[scale=2] at (200bp,200bp) {$...$};
  \node (x_2kminus1) at (230bp,200bp) [draw, circle] {};
  \node  at (227bp,210bp) {$x_{2k-1}$};
  \node (x_2k) at (270bp,200bp) [draw, circle, fill=black] {};
  \node  at (267bp,210bp) {$x_{2k}$};
  \node (x1) at (90bp,110bp) [draw, circle, fill=black] {};
  \node  at (87bp,100bp) {$x_1$};
  \node (x2) at (130bp,110bp) [draw, circle, fill=black] {};
  \node  at (127bp,100bp) {$x_2$};
  \node (x3) at (170bp,110bp) [draw, circle, fill=black] {};
  \node  at (167bp,100bp) {$x_3$};
  \node[scale=2] at (200bp,110bp) {$...$};
  \node (x_kminus1) at (230bp,110bp) [draw, circle, fill=black] {};
  \node  at (227bp,100bp) {$x_{k-1}$};
  \node (x_k) at (270bp,110bp) [draw, circle] {};
  \node  at (267bp,100bp) {$x_k$};
  \draw [-] (x_kplus1) -- node[above,sloped] {} (x_k);
  \draw [-] (x_kplus1) -- node[above,sloped] {} (x2);
  \draw [-] (x_kplus1) -- node[above,sloped] {} (x3);
  \draw [-] (x_kplus1) -- node[above,sloped] {} (x_kminus1);
  \draw [-] (x_kplus2) -- node[above,sloped] {} (x_kminus1);
  \draw [-] (x_kplus2) -- node[above,sloped] {} (x_k);
  \draw [-] (x_kplus2) -- node[above,sloped] {} (x3);
  \draw [-] (x_kplus3) -- node[above,sloped] {} (x_kminus1);
  \draw [-] (x_kplus3) -- node[above,sloped] {} (x_k);
  \draw [-] (x_2kminus1) -- node[above,sloped] {} (x_k);
  \node at (360bp,200bp) {Clique on $\{x_{k+1},...,x_{2k}\}$};
\draw[line width=1pt, draw = black] (175bp,203bp) ellipse(120bp and 25bp) node{};
  \node at (355bp,110bp) {Clique on $\{x_1,...,x_k\}$};
  \draw[line width=1pt, draw = black] (175bp,108bp) ellipse(120bp and 25bp) node{};
  \end{tikzpicture}}
  \end{center}
  \caption{The graphs $A_{k}$ have location-domination number
    half their order.}
\label{fig:A_k}
\end{figure}

\begin{prop}\label{prop:A_k}
For any $k\ge  2$, $\gL(A_k) = k = n(A_k)/2$.
\end{prop}
\begin{proof}
The set $\{x_1,\ldots,x_{k-1}\}\cup\{x_{2k}\}$ is a locating-dominating set of $A_k$, and so $\gL(A_k) \le k$.
For the other direction, let $D$ be a locating-dominating set of $A_k$, and let $A=\{x_1,\ldots,x_k\}$ and
$B=\{x_{k+1},\ldots,x_{2k}\}$. Let $D_A = D \cap A$ and let $D_B = D \cap B$. Further, let $a = |D_A|$ and let $b = |D_B|$. In order to dominate the vertex $x_1$ (respectively, $x_{2k}$), we note that $D_A \ne \emptyset$ (respectively, $D_B \ne \emptyset$).
For every two vertices $x_i$ and $x_j$ in $A$ with $i<j$, we have $N(x_i)\subseteq N(x_j)$. Since $A$ is a clique, every vertex of $A$ is dominated by each vertex of $D_A$. Further, if $v \in A$, then either $N(v) \cap B = \emptyset$ or $v$ is adjacent to consecutive vertices of $D_B$ in the sense that if $x_j$ is a vertex in $D_B$ of largest subscript adjacent to $v$, then $v$ is adjacent to all vertices $x_\ell \in D_B$ with $\ell \le j$. Hence, the vertices in $A$ have at most $b+1$ distinct neighborhoods in $D_B$ (including the possibility of a vertex in $A$ having an empty neighborhood in $D_B$).
By a similar argument, the vertices in $B$ have at most $a+1$ distinct neighborhoods in $D_A$. However, there can be only one vertex in $V(A_k) \setminus D$ which is dominated by every vertex in $D$, implying that
\[
2k - |D| = |V(A_k)\setminus D|\le  (a+1)+(b+1)-1=a+b+1=|D|+1,
\]
or, equivalently, $2k \le 2|D|+1$, and so $|D| \ge \lceil\frac{2k-1}{2}\rceil = k$. Since $D$ is an arbitrary locating-dominating set of $A_k$, this implies that $\gL(A_k) \ge k$. Consequently, $\gL(A_k) = k = n(A_k)/2$.
\end{proof}

Similar as in~\cite{fgknpv11}, we will show how to combine the graphs of Definition~\ref{def:Ak} to obtain more extremal examples. If the cardinality of a minimum locating set in a graph $G$ is equal to $\gL(G)$, then we say that $G$ is \emph{combinable}. Given two graphs $G$ and $H$, the \emph{complete join} of $G$ and $H$, abbreviated $G\bowtie H$, is the graph obtained from the disjoint union of $G$ and $H$ by adding all the edges $uv$ with $u\in V(G)$ and $v\in V(H)$. The complete join of a set $\{G_1,\ldots,G_k\}$ of more than two graphs is the graph obtained from the disjoint union of $G_1,\ldots, G_k$ by adding all the edges $uv$ with $u\in V(G_i)$ and $v\in V(G_j)$, for $1\leq i<j\leq k$. 

\begin{lem}\label{lem:combi}
If $G$ and $H$ are two combinable graphs, then $\gL(G\bowtie H) \ge \gL(G)+\gL(H)$.
\end{lem}
\begin{proof}
Note that a vertex in $V(G)$ cannot locate any pair in $V(H)$ (and vice-versa), however it may help to dominate some vertex. Therefore, for any locating-dominating set $D$ of $G\bowtie H$, $D\cap V(G)$ must be a locating set of $G$, and $D\cap V(H)$, a locating set of $H$. Then, by the definition of a combinable graph, $|D\cap V(G)|\ge \gL(G)$ and $|D\cap V(H)|\ge \gL(H)$, which completes the proof.
\end{proof}

\begin{thm}
Let $\cA$ be a set of vertex-disjoint graphs where each member of $\cA$ is isomorphic to some graph $A_{k}$ ($k \ge  2$). Let $G(\cA)$ be the complete join of all members in $\cA$. Then, $G$ has location-domination number half its order.
\end{thm}
\begin{proof} We use induction on the size of $\cA$, by proving the following claim: every graph $G(\cA)$ is combinable, has location-domination number half its order, and there is a minimum locating-dominating set where no vertex is dominated by every vertex in the set. The proof of Proposition~\ref{prop:A_k} shows that for $k\ge  2$, the graph $A_k$ is combinable, has location-domination number half its order, and there is a minimum locating-dominating set (namely, the set $\{x_1,\ldots,x_{k-1}\}\cup\{x_{2k}\}$) where no vertex is dominated by every vertex in the set. Hence if $|\cA|=1$, we are done. This establishes the base case.

Now, assume $|\cA|>1$. Let $G=G(\cA)$, and let $A_i$ be some member of $\cA$. By induction, the claim is true for $G_1=G(\cA\setminus A_i)$ and for $G_2=G(A_i)$. Hence by Lemma~\ref{lem:combi}, $\gL(G)\ge  \gL(G_1)+\gL(G_2)=\frac{|V(G)|}{2}$. Moreover, consider two minimum locating-dominating sets $D_1$ of $G_1$ and $D_2$ of $G_2$, where no vertex of $V(G_i) \setminus D_i$ is dominated by every vertex of $D_i$ in $G_i$ for $i \in \{1,2\}$. Then, $D=D_1\cup D_2$ is a dominating set and there is no vertex of $G$ dominated by every vertex of $D$. All vertex pairs within one of the two subgraphs are located, and finally, each pair $u,v$ with $u\in V(G_1)$ and $v\in V(G_2)$ is located by the vertex of $D_1$ that does not dominate $u$. It remains to show that $G$ is combinable. Let $L$ be a minimum locating set of $G$. By a similar argument as in the proof of Lemma~\ref{lem:combi}, $L\cap V(G_1)$ must be a locating set of $G_1$ and $L\cap V(G_2)$ must be a locating set of $G_2$. Then by induction we know that $|L\cap V(G_1)|\ge  \gL(G_1)$ and $|L\cap V(G_2)|\ge  \gL(G_2)$, implying that $\gL(G) \le \gL(G_1)+\gL(G_2)\le  |L|$. Since also $|L|\le  \gL(G)$, we have equality and $G$ is combinable.
\end{proof}

\subsection{A family of twin-free graphs with large domination number and location-domination number half the order}

Let $G$ be a graph with $\gL(G) = \frac{|V(G)|}{2}$. If $G$ contains a vertex $x$ such that, when we identify $x$ with a vertex of some other graph $H$ that is vertex-disjoint from $G$, every locating-dominating set of the resulting graph contains at least half of the vertices of $G$, then we say that $G$ is \emph{attachable} and $x$ is a \emph{link vertex} of $G$. Examples of attachable graphs are paths on two vertices, as well as any graph obtained from a star where one edge is subdivided twice, and every other edge is subdivided once (the link vertex is the center of the star). See Figure~\ref{fig:attach} for an illustration, where the darkened vertices form a minimum locating-dominating set of the graph.

\begin{figure}[htb]
\tikzstyle{every node}=[circle, draw, fill=black!0, inner sep=0pt,minimum width=.16cm]
\begin{center}
\begin{tikzpicture}[thick,scale=.6]
  \draw(0,0) { 
    +(0.00,3.50) -- +(0.00,2.00)
    +(2.00,3.50) -- +(1.50,2.00)
    +(1.50,2.00) -- +(1.50,1.00)
    +(1.50,1.00) -- +(1.50,0.00)
    +(2.50,1.00) -- +(2.50,2.00)
    +(2.50,2.00) -- +(2.00,3.50)
    +(4.00,3.50) -- +(4.00,2.00)
    +(5.50,3.50) -- +(5.50,2.00)
    +(5.50,2.00) -- +(5.50,1.00)
    +(5.50,1.00) -- +(5.50,0.00)
    +(7.00,0.00) -- +(7.00,1.00)
    +(7.00,2.00) -- +(7.00,1.00)
    +(7.00,2.00) -- +(8.00,3.50)
    +(8.00,3.50) -- +(7.50,2.00)
    +(7.50,2.00) -- +(7.50,1.00)
    +(8.00,1.00) -- +(8.00,2.00)
    +(8.00,2.00) -- +(8.00,3.50)
    +(8.50,2.00) -- +(8.50,1.00)
    +(8.50,2.00) -- +(8.00,3.50)
    +(8.00,3.50) -- +(9.00,2.00)
    +(9.00,2.00) -- +(9.00,1.00)
    +(9.00,1.00) node{}
    +(9.00,2.00) node[fill=black!100]{}
    +(8.50,2.00) node[fill=black!100]{}
    +(8.50,1.00) node{}
    +(8.00,1.00) node{}
    +(8.00,2.00) node[fill=black!100]{}
    +(7.50,2.00) node[fill=black!100]{}
    +(7.50,1.00) node{}
    +(8.00,3.50) node[fill=black!100]{}
    +(7.00,2.00) node{}
    +(7.00,1.00) node[fill=black!100]{}
    +(7.00,0.00) node{}
    +(5.50,2.00) node{}
    +(5.50,3.50) node[fill=black!100]{}
    +(5.50,1.00) node[fill=black!100]{}
    +(5.50,0.00) node{}
    +(4.00,2.00) node{}
    +(4.00,3.50) node[fill=black!100]{}
    +(2.50,2.00) node{}
    +(2.50,1.00) node[fill=black!100]{}
    +(2.00,3.50) node[fill=black!100]{}
    +(1.50,2.00) node{}
    +(1.50,1.00) node[fill=black!100]{}
    +(1.50,0.00) node{}
    +(0.00,3.50) node[fill=black!100]{}
    +(0.00,2.00) node{}
    +(3,4) node[rectangle, draw=white!0, fill=white!100]{$H$}   
  };
  \draw[line width=1pt] (4.1,3.6) ellipse(5.5cm and 1.0cm);

\end{tikzpicture}
\end{center}
\vskip -0.6 cm \caption{A graph with location-domination number half its order obtained from a set of attachable graphs.} \label{fig:attach}
\end{figure}
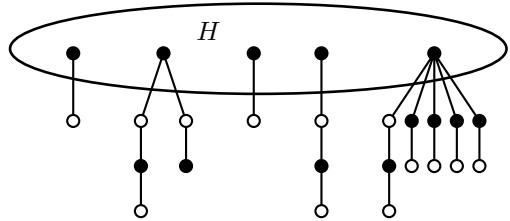

As an immediate consequence of this definition, we have the following observation.

\begin{ob}\label{obm:hang}
If $G$ is a graph obtained from the disjoint union of a graph $H$ and $|V(H)|$ disjoint attachable graphs, by identifying each vertex of $H$ with a link vertex of one of the attachable graphs, then $\gL(G)\ge \frac{|V(G)|}{2}$.
\end{ob}

We note that in Observation~\ref{obm:hang}
if $H$ is a graph without isolated vertices, then $G$ is twin-free.
Hence, $G$ is extremal with respect to the bound of Conjecture~\ref{conj}.

\subsection{Extremal trees}\label{sec:trees}

Recall that by Proposition~\ref{prop:VC}, Conjecture~\ref{conj} holds for bipartite graphs (and hence trees). We now characterize all trees that are extremal with respect to the bound of Conjecture~\ref{conj}.

\begin{defin}\label{def:T}
Let $\cT$ be the family of trees $T$ satisfying the following properties: \1 \\
\indent $\bullet$ $T$ has a perfect matching $M$. \1 \\
\indent $\bullet$ Each edge of $M$ has one end colored  white and the other end colored black. \1 \\
\indent $\bullet$ Each white vertex is either a leaf, or has degree~2 and is adjacent to a black vertex that \\ \hspace*{0.55cm} has a white leaf as a neighbor.
\end{defin}

A tree from Definition~\ref{def:T} is illustrated in Figure~\ref{fig:extr-tree}, where the thick edges belong to the matching.

\begin{figure}[ht]
\centering
  \scalebox{1.0}{\begin{tikzpicture}[join=bevel,inner sep=0.6mm,line width=0.8pt, scale=0.4]

\path (0,0) node[draw, shape=circle,fill] (b1) {};
\path (b1)+(-2,0) node[draw, shape=circle] (w1) {};

\path (b1)+(0,1.5) node[draw, shape=circle,fill] (b2) {};
\path (b2)+(-2,0) node[draw, shape=circle] (w2) {};

\path (b1)+(2,1.5) node[draw, shape=circle,fill] (b3) {};
\path (b3)+(2,0) node[draw, shape=circle] (w3) {};

\path (b1)+(2,0) node[draw, shape=circle] (w4) {};
\path (w4)+(2,0) node[draw, shape=circle,fill] (b4) {};

\path (b1)+(2,-1.5) node[draw, shape=circle] (w5) {};
\path (w5)+(2,0) node[draw, shape=circle,fill] (b5) {};

\path (w3)+(2,0.75) node[draw, shape=circle,fill] (b6) {};
\path (b6)+(2,0) node[draw, shape=circle] (w6) {};

\path (b6)+(0,-1.5) node[draw, shape=circle,fill] (b7) {};
\path (b7)+(2,0) node[draw, shape=circle] (w7) {};

\path (b5)+(2,0.75) node[draw, shape=circle,fill] (b8) {};
\path (b8)+(2,0) node[draw, shape=circle] (w8) {};

\path (b5)+(2,-0.75) node[draw, shape=circle,fill] (b9) {};
\path (b9)+(2,0) node[draw, shape=circle] (w9) {};

\path (w9)+(2,0) node[draw, shape=circle,fill] (b10) {};
\path (b10)+(2,0) node[draw, shape=circle] (w10) {};

\draw[line width=2.5pt] (b1)--(w1) (b2)--(w2) (b3)--(w3) (b4)--(w4) (b5)--(w5)
                        (b6)--(w6) (b7)--(w7) (b8)--(w8) (b9)--(w9) (b10)--(w10);

\draw[line width=0.5pt] (b1)--(b2) (b1)--(b3) (b1)--(w4) (b1)--(w5)
                        (b6)--(w3) (b6)--(b7) (b8)--(b5)--(b9) (w9)--(b10);

  \end{tikzpicture}}
  \caption{A tree satisfying Definition~\ref{def:T}.}
  \label{fig:extr-tree}
\end{figure}
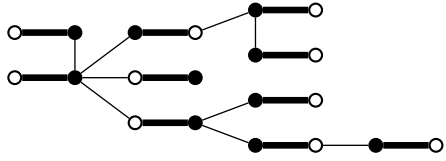

\medskip
\begin{prop}\label{prop:trees}
Every tree $T\in\cT$ with order $n$ satisfies $\gL(T)=\frac{n}{2}$.
\end{prop}
\begin{proof}
The upper bound follows from Proposition~\ref{prop:VC}. In fact, one can check that the set of black vertices forms a locating-dominating set of $T$. For the lower bound, let $T \in \cT$ have order~$n$ with $M$ its perfect matching and a black-white coloring of $T$ as defined in Definition~\ref{def:T}. If $n = 2$, then the desired result is immediate. Hence, we may assume that $n \ge 4$. Let $D$ be an arbitrary locating-dominating set of $T$. If we are able to associate a distinct vertex, $f(e)$, of $D$ to each edge $e$ of $M$, then $|D| \ge |M| = \frac{n}{2}$ and we will be done. Let $e=uv \in M$. Renaming $u$ and $v$ if necessary, we may assume that $u$ is colored white in the black/white coloring of the vertices of $T$. If both $u$ and $v$ belong to $D$, we choose $f(e) = v$ (and note that $v$ is colored black and is therefore not a leaf in $T$). If exactly one of $u$ and $v$ belongs to $D$, we choose $f(e)$ to be the end of $e$ that belongs to $D$. Suppose that neither $u$ nor $v$ belongs to $D$. Recall that $u$ is colored white. If $u$ is a leaf, then $u$ would not be dominated by $D$, a contradiction. Hence, by Definition~\ref{def:T}, $d(u)=2$ and $u$ has a black neighbor $b$ adjacent to a white leaf $w$. Necessarily, $bw \in M$. In order to dominate $u$, the vertex $b$ must belong to $D$. If $w \notin D$, then both $u$ and $w$ are adjacent to $b$ but to no other vertex of $D$, and so $u$ and $w$ would not be located by $D$, a contradiction. Hence, $w$ must belong to $D$. Since both ends of the edge $bw$ belong to $D$ and $b$ is colored black, we chose $f(bw) = b$. We now choose $f(e) = w$. We claim that there is no edge $e' \in M$ other than $e$, with $f(e')=w$. Indeed, if this was the case, then the neighbor of $b$ in $e'$ would not be located from $u$, a contradiction. This completes the proof.
\end{proof}

\medskip
The following result characterizes all trees with location-domination number one-half their order.

\begin{thm}
Let $T$ be a tree of order $n\ge  2$ without any open twins. Then, $\gL(T)=\frac{n}{2}$ if and only if $T\in\cT$.
\end{thm}
\begin{proof}
Proposition~\ref{prop:trees} shows that every member $T$ of $\cT$ with order $n$ satisfies $\gL(T)=\frac{n}{2}$.

For the other direction, we proceed by induction on the order $n \ge 2$ of a tree $T$ without any open twins satisfying $\gL(T)=\frac{n}{2}$. If $n=2$, then $T = K_2 \in \cT$ and the claim is clearly true. This establishes the base case. Let $n \ge 4$ be even and assume that if $T'$ is a tree of order~$n'<n$ without any open twins satisfying $\gL(T')=\frac{n'}{2}$, then $T' \in \cT$. Let $T$ be a tree of order $n \ge 4$ without any open twins satisfying $\gL(T)=\frac{n}{2}$ and consider a longest path in $T$ that connects some leaf $r$ and a
leaf $u$. Since $T$ has no open twins, the neighbor $v$ of $u$ has degree~$2$. Let $w$ be the other neighbor of $v$. Let $T' = T - \{u,v\}$ be the tree obtained from $T$ by removing the vertices $u$ and $v$. By Proposition~\ref{prop:VC}, we have $\gL(T')\le \frac{n-2}{2}$. Let $D'$ be a minimum locating-dominating set of $T'$. Then the set $D' \cup \{v\}$ is a locating-dominating set of $T$, implying that $\frac{n}{2} = \gL(T) \le \gL(T') + 1 \le \frac{n-2}{2} + 1 = \frac{n}{2}$. Consequently, we must have equality throughout this inequality chain. In particular, $\gL(T') =  \frac{n-2}{2}$. Applying the inductive hypothesis to $T'$, we have $T'\in \cT$. Let $M'$ be the perfect matching of $T'$ and consider the associated black-white coloring, $\cC'$, of the vertices according to Definition~\ref{def:T}. Let $M = M' \cup \{uv\}$, and note that $M$ is a perfect matching of $T$.

If $w$ is colored black, then we extend the coloring $\cC'$ by coloring $v$ black and coloring $u$ white. Then, $T$ belongs to $\cT$, as desired. Hence we may assume that $w$ is colored white, for otherwise we are done.  We distinguish two cases.

\noindent {\bf Case 1: \boldmath{$w$} is a leaf in $T'$.} Let $x$ be its (black) neighbor in $T'$. If $w$ is the only white neighbor of $x$ in $T'$, then we extend the coloring $\cC'$ by coloring $v$ black and coloring $u$ white. Then, $T$ belongs to $\cT$. Hence, we may assume that $x$ has a white neighbor in $T'$ different from $w$. Let $a$ be such a neighbor of $x$. Since $T' \in \cT$, the vertex $a$ has degree~2 in $T'$. Let $b$ be the vertex matched to $a$ by
$M'$. We show that every neighbor of $b$ different from $a$ is colored black in $\cC'$. Suppose, to the contrary, that $b$ has a white neighbor $c$ different from $a$. Then, $c$
cannot be a leaf since $b$ is already matched to $a$ by $M'$. Hence, $c$ is matched to a black vertex, $d$ say, by $M'$. By definition, $d$ has to be adjacent to a white leaf. But then this leaf is not matched by $M'$ since its only neighbor, $d$, is already matched by $M'$ to $c$, a contradiction. Therefore, every neighbor of $b$ different from $a$ is colored black in $\cC'$.

If $b$ has at least two black neighbors, then the set consisting of all black vertices in $T'$ different from $b$, together with $v$ forms a locating-dominating set of $T$ (where $b$ is uniquely located by its black neighbors) of cardinality $\frac{n-2}{2}$, a contradiction. Therefore, $b$ has at most one black neighbor, implying that $d(b) \le 2$.

If $d(b)=1$, we can recolor $b$ white and $a$ black, and we can extend the resulting coloring to $T$ by coloring $v$ black and coloring $u$ white. In this case, $T$ clearly belongs to $\cT$. Hence we may assume that $d(b)=2$, for otherwise the desired result follows.
Let $c$ be the neighbor of $b$ distinct from $a$. By the earlier observations, $c$ is colored black. Let $d$ be the (white) neighbor of $c$ matched by $M$.  If $d$ is not a leaf, then let $e$ be the (black) neighbor of $d$ different from $c$. In this case, the set consisting of all black vertices in $T'$ different from $b$, together with $v$ forms a locating-dominating set of $T$ (where $d$ is uniquely located by $c$ and $e$, and $b$ is uniquely located by $c$) of cardinality $\frac{n-2}{2}$, a contradiction. Hence, $d$ is a leaf and we can recolor $a$ black and $b$ white and as before extend the resulting coloring to $T$ to show that $T$ belongs to $\cT$.

\noindent {\bf Case 2: \boldmath{$d_{T'}(w)=2$}.} Let $x$ be the neighbor of $w$ matched by $M$, let $a$ be its other (black) neighbor, and let $b$ the (white) neighbor of $a$ matched by $M$. By definition of $\cT$, the vertex $b$ is a leaf and every neighbor of $x$, if any, different from $w$, is colored black.
If $d(x) \ge 3$, then the set consisting of all black vertices in $T'$ different from $x$, together with $v$ forms a locating-dominating set of $T$ of cardinality $\frac{n-2}{2}$, a contradiction. Hence, $d(x) \le 2$.

If $d(x) = 1$, then we can recolor $x$ white and $w$ black, and we can extend the resulting coloring to $T$ by coloring $v$ black and coloring $u$ white. In this case, $T$ clearly belongs to $\cT$.
If $d(x)=2$, let $c$ be the black neighbor of $x$ and let $d$ be the white neighbor of $c$ matched to $c$ by $M$. If the vertex $d$ is not a leaf, then as before the set consisting of all black vertices in $T'$ different from $x$, together with $v$ forms a locating-dominating set of $T$ (where $d$ is located by $c$ and its other black neighbors and $x$ is located by $c$ only) of cardinality $\frac{n-2}{2}$, a contradiction. Hence, $d$ is a leaf and we can
recolor $x$ white and $w$ black and extend the coloring to $T$. Once again, $T$ clearly belongs to $\cT$.
\end{proof}

\section{Co-bipartite and split graphs}\label{sec:cobip-split}

Since any split graph or co-bipartite graph $G$ has either independence number or clique number at least half its order, the results in~\cite{conjpaper} mentioned in the introduction imply that $\gL(G)\le \lfloor\frac{n}{2}\rfloor+1$ if $G$ is twin-free and without isolated vertices, where $n$ is the order of $G$. We are able to slightly improve this bound, therefore proving Conjecture~\ref{conj} for these classes.

\begin{thm}
Let $G$ be a twin-free graph of order $n$ with no isolated vertices. If $G$ is a co-bipartite or split graph, then $\gL(G)\le \frac{n}{2}$.
\end{thm}
\begin{proof}
\noindent\textbf{Co-bipartite graphs.} If $G$ is co-bipartite, let $X$ and $Y$ be the two cliques partitioning $V(G)$, with $|X| \le |Y|$. If every vertex of $Y$ has a neighbor in $X$, then since $G$ has no closed twins, every vertex in $Y$ has a nonempty and distinct neighborhood within $X$, implying that $X$ is a locating-dominating set of $G$. Thus, in this case, $\gL(G) \le |X| \le \frac{n}{2}$. Hence we may assume that $Y$ has a vertex $y$ with no neighbor in $X$. Since $G$ has no closed twins, such a vertex $y$ is unique.

If there is a vertex $x\in X$ that has no neighbor in $Y$, since $G$ is twin-free, $x$ is unique and the set $(X\setminus\{x\}) \cup\{y\}$ is a locating-dominating set of $G$, and once again $\gL(G) \le |X| \le \frac{n}{2}$. Hence, we may assume that every vertex in $X$ has a neighbor in $Y$. Since $G$ has no closed twins, every vertex in $X$ has a nonempty and distinct neighborhood within $Y$, implying that $X \cup \{y\}$ is a locating-dominating set of $G$. Hence, if $|Y|-|X|\ge 2$, then $\gL(G) \le |X| + 1 \le \frac{n}{2}$. We may therefore assume that $|Y|-|X|\le  1$, for otherwise we are done. If there is no vertex in $X$ that is adjacent to every vertex of $Y \setminus \{y\}$, then $Y \setminus \{y\}$ is a locating-dominating set of $G$, and so $\gL(G) \le |Y|-1 \le |X| \le \frac{n}{2}$. If there is a vertex $x$ in $X$ that is adjacent to every vertex of $Y \setminus \{y\}$, then $(X \setminus \{x\}) \cup\{y\}$ is a
locating-dominating set of $G$, and so $\gL(G) \le |X| \le \frac{n}{2}$. In both cases, $\gL(G) \le \frac{n}{2}$, as desired.

\noindent\textbf{Split graphs.}  The proof for the case of split graphs is similar to that for
co-bipartite graphs. Assume $G$ is a split graph, and let $X$ be a clique and $Y$ be an independent set that form a partition of $V(G)$. We note that $X$ is a vertex cover of $G$. Suppose that $|Y|\ge \frac{n}{2}$. By Proposition~\ref{prop:VC}, the set $X$ is a locating-dominating set of $G$, and so $\gL(G) \le |X| = n - |Y| \le \frac{n}{2}$. Hence, we may assume that $|Y|<\frac{n}{2}$, for otherwise the desired result follows.

If every vertex of $X$ has a neighbor in $Y$, then since $G$ has no closed twins, every vertex in $X$ has a nonempty and distinct neighborhood within $Y$, implying that $Y$ is a locating-dominating set of $G$. Thus, in this case, $\gL(G) \le |Y| <\frac{n}{2}$. Hence we may assume that $X$ has a unique vertex $x$ with no neighbor in $Y$. We note that in this case, the set $Y \cup\{x\}$ is a locating-dominating set of $G$, and so $\gL(G) \le |Y| + 1$.
If $|X|-|Y|\ge 2$, then $|Y| + 1 \le n/2$, and the desired result follows. Hence, we may assume that $|X|-|Y|\le 1$.
If there is no vertex in $Y$ that is adjacent to every vertex of $X \setminus \{x\}$, then $X \setminus \{x\}$ is a locating-dominating set of $G$, and so $\gL(G) \le |X|-1 \le |Y| < \frac{n}{2}$.
If there is a vertex $y$ in $Y$ that is adjacent to every vertex of $X \setminus \{x\}$, then $(Y \setminus \{y\}) \cup\{x\}$ is a
locating-dominating set of $G$, and so $\gL(G) \le |Y| < \frac{n}{2}$. In both cases, $\gL(G) < \frac{n}{2}$, as desired.
\end{proof}

\section{Conclusion}

Though we have advanced the study of Conjecture~\ref{conj}, it remains wide open. It would be interesting to extend the result of~\cite{conjpaper} for bipartite graphs, to the class of triangle-free graphs. We also raise the question of extending the characterization of extremal trees to the other (structured) classes for which the conjecture is known to hold, including bipartite graphs, split graphs, and co-bipartite graphs.

\end{document}